\documentclass[11pt]{amsart}

\usepackage[margin=2.5cm]{geometry}

\usepackage[utf8]{inputenc}

\usepackage{amsmath}
\usepackage{amssymb}
\usepackage{amsthm}
\usepackage{enumitem}
\setlist[enumerate]{label=(\arabic*)}

\newcommand{\R}{\mathbb{R}}

\newtheorem{definition}{Definition}[section]

\newtheorem{remark}[definition]{Remark}
\newtheorem{lemma}[definition]{Lemma}

\newtheorem*{MainTheorem}{Main Theorem}

\title{On a generalization of a theorem of S. Bernstein}

\author{J. Ciesielski}
\author{J. Maksymiuk}
\author{M. Starostka}

\address{
\noindent Faculty of Applied Physics and Mathematics\\
Gda\'{n}sk University of Technology\\
Narutowicza 11/12, 80-233 Gda\'{n}sk, Poland
}

\email{jakub.ciesielski@pg.edu.pl, jakub.maksymiuk@pg.edu.pl, maciej.starostka@pg.edu.pl}

\thanks{M. Starostka was partially supported by Grants Beethoven2 and Preludium9 of the
National Science Centre, Poland, no. 2016/23/G/ST1/04081 and no. 2015/17/N/ST1/02527}

\begin{document}

\begin{abstract} 
In this paper we obtain a solution to the second order boundary value problem of the form
$\frac{d}{dt}\Phi'(\dot{u})=f(t,u,\dot{u}),\ t\in[0,1],\ u\colon\R\to\R$ with
Dirichlet and Sturm-Liouville boundary conditions, where $\Phi\colon\R\to\R$ is
strictly convex, differentiable function and
$f\colon[0,1]\times\R\times\R\to\R$ is continuous and satisfies a suitable
growth condition. Our result is based on a priori bounds for the solution and homotopical invariance
of the Leray-Schauder degree.
\end{abstract}


\maketitle

\section{Introduction}

Our purpose is to show the existence of solutions to second order boundary value problems of the
form
\begin{equation}
\label{eq:problem}\tag{P}
\begin{cases}
\frac{d}{dt}\Phi'(\dot{u})=f(t,u,\dot{u}),\ t\in[0,1]
\\
u\in (BC)
\end{cases}
\end{equation}
where $u\in (BC)$ means that $u$ satisfies either Dirichlet or Sturm-Liouville boundary conditions,
$\Phi'$ is an increasing homeomorphism satisfying some technical assumptions and $f$ is a continuous
function satisfying suitable growth conditions.

In particular,  $\Phi(x)= \frac{1}{p_1}|x|^{p_1} + \ldots \frac{1}{p_n}|x|^{p_n}$, $1 < p_i \leq 2$
is in the considered class of functions and if $n = 1$ then the differential operator on the left
hand side of the equation is a $p$-Laplacian.

To prove the existence we use topological methods. This approach was already used by many authors.
In \cite{GraGueLee78} and \cite{FriORe88} the authors consider the case of a Laplace operator with
various
boundary conditions. Generalizations to the $p$-Laplacian and to the operator defined by an
arbitrary
increasing homeomorphism were developed in \cite{KelTer13} and \cite{MaZhaLiu15}, respectively.
However, in \cite{KelTer13} and \cite{MaZhaLiu15} authors subject the equation to very specific
boundary conditions, namely $u(0) = A$, $\dot{u}(1) = B$. In order to show the existence for
Dirichlet and general Sturm-Liouville conditions more effort has to be put in as can be seen below.
Moreover, we consider different assumptions on the function $f$.

The main idea in the orginal paper \cite{GraGueLee78} was to use the topological transversality
theorem.
This is a fixed point type theorem (see \cite{DugGra10}). We decided to use an approach via
Leray-Schauder degree theory instead since it is essentially equivalent but the degree theory is
familiar to a broader audience.

\section*{Acknowledgements}
The authors would like to thank Professor Andrzej Granas for suggesting this problem during the
seminar in Gdańsk in 2016.

\section{Preliminaries}

In this section we are more precise on the assumptions on functions $\Phi$ and $f$ occurring in the
problem \ref{eq:problem}. We also state the main theorem.

We assume that $\Phi\colon \R \to \R$ satisfies
\begin{enumerate}[label=$(\Phi_\arabic*)$,ref=($\Phi_\arabic*$)]
\item \label{asm:Phi:convex} $\Phi$ is strictly convex, differentiable and $\Phi(x)/|x|\to \infty$
as $|x|\to \infty$,
\item $\Phi(0) = \Phi'(0) = 0$,
\item \label{asm:Phi:psi_differentiable}
$(\Phi')^{-1}$ is continuously differentiable,
\item
\label{asm:Phi:nabla2}
there exists a constant $k_\Phi>1$ such that
\[
k_\Phi\Phi(x)\leq \Phi'(x) x\quad \text{for all $x\in \R$}
\]
\end{enumerate}

Assumption \ref{asm:Phi:convex} guarantees that $\Phi'$ is an increasing homeomorphism and so
$(\Phi')^{-1}$ exists. However, we will also need that it is continuously differentiable
\ref{asm:Phi:psi_differentiable}. We put $\varphi = \Phi'$, $\psi = \varphi^{-1}$ and $L_\varphi u =
\frac{d}{dt}\varphi(\dot{u})$. The domain of $L_\varphi$ will be defined later. As already mentioned
in
the introduction, $\Phi(x) = |x|^p$ satisfies the above assumptions and in this case $L_\varphi$ is
just
a $p$-Laplacian. A more general example of $\Phi$ is provided by an N-function satisfying the
$\nabla_2$-condition (see. \cite{KraRut61}).

We assume that $f\colon [0,1]\times \R\times \R \to \R$ is continuous and satisfies
\begin{enumerate}[label=$(f_\arabic*)$, ref=$(f_\arabic*)$]
\item
\label{asm:f:geometric}
there exists a constant $R>0$ such that
\[
x f(t,x,0) > 0 \text{ for $|x|>R$}
\]
\item
\label{asm:f:growth}
there exist positive functions $S$, $T$, bounded on bounded sets such that
\[
|f(t,x,v)|\leq S(t,x) (\Phi'(v)\cdot v-\Phi(v))  + T(t,x)
\]
\end{enumerate}
We consider the following boundary conditions:

\begin{enumerate}[ref=BC$_{\arabic*}$,label=(\arabic*)]
\item \label{BC:Dirichlet-nh}%
Dirichlet
\begin{equation*}
\tag{\ref{BC:Dirichlet-nh}}
    u(0)=A,\ u(1)=B,
\end{equation*}
\item \label{BC:S-L_I-nh}%
Sturm-Liouville
\begin{gather*}
\tag{\ref{BC:S-L_I-nh}}
    -\alpha u(0) + \beta \dot{u}(0)=A,\quad \alpha,\beta >0\\
    a u(1) + b \dot{u}(1)=B,\quad a,b >0.
\end{gather*}
\end{enumerate}
The purpose of the paper is to prove the following existence result.

\begin{MainTheorem}
Suppose $\Phi$ and $f$ satisfy $(\Phi_1)$-$(\Phi_4)$ and $(f_1)$ - $(f_2)$ respectively. Then under
boundary conditions \eqref{BC:Dirichlet-nh} or \eqref{BC:S-L_I-nh} the problem $(\ref{eq:problem})$
has
at least one solution.
\end{MainTheorem}

\section{Proof of the main theorem}
Fix $\Phi$, $f$ and boundary conditions $(BC)$. We will now show that the existence of a solution to
$(\ref{eq:problem})$ is equivalent to the existence of a fixed point of some map on a Banach space.
Let $\hat{K}\colon C^0([0,1]) \times \R \times \R \to C^1([0,1])$ be given by
\[
\hat{K}(v,c_1,c_2)(t) = c_1 + \int_0^t \psi(\int_0^\tau v(s) \, ds + c_2)\, d\tau
\]
For every $v$ we would like to choose $c_1$ and $c_2$ in such a way that $u = \hat{K}(v,c_1,c_2)$
is an element of $C^1_{BC}$, i.e. it satisfies boundary conditions. Moreover, we need that $c_1$ and
$c_2$ depend continuously on $v$.

\begin{remark}
Note that this trivializes in  \cite{KelTer13,MaZhaLiu15}. For boundary conditions considered
therein $c_1$ and $c_2$ are constants independent of $v$. We cannot proceed in such a way here.
\end{remark}

\begin{lemma}\label{lemma:implicitFunction}
Let $X$ be a metric space and let $G\colon X \times \R \to \R$ be continuous.
Suppose that
\begin{enumerate}
\item for every $v \in X$ function $g_v(\cdot) = G(v,\cdot)\colon \R \to \R$ is an increasing
homeomorphism,
\item if $\{v_n\}$ is bounded and $b_n \to \pm\infty$ then $G(v_n,b_n) \to \pm\infty$.
\end{enumerate}
Fix a constant $C \in \mathbb {R}$.  Then the function $c\colon X \to \R$ defined by $G(v,c(v)) = C$
is continuous.
\end{lemma}

Note that if $g_v$ is differentiable and $g_v'$ is positive then the conclusion follows from
implicit function theorem. However, in the problem that we consider $g'_v$ is only non-negative.

\begin{proof}
Suppose, to derive a contradiction, that $v_n \to v_0$ and $c_n := c(v_n)$ does not converge to
$c_0 := c(v_0)$, i.e. there exists $\epsilon > 0$ such that, up to subsequence, $|c_n - c_0| >
\epsilon$. By $(2)$, the sequence $c_n$ is bounded so it converges, again up to subsequence, to some
$c_0' \neq c_0$. By the continuity and injectivity of $G$,
\[
C = G(v_n,c_n) \to G(v_0,c_0') \neq G(v_0,c_0) = C
\]
\end{proof}
We use this abstract lemma for our problem.

\begin{lemma}\label{lemma:functionK}
Fix one of boundary conditions (\ref{BC:Dirichlet-nh}) or (\ref{BC:S-L_I-nh}). Then for every $v
\in C^0([0,1])$ there exist unique constants $c_1(v),c_2(v)$ such that $\hat{K}(v,c_1(v),c_2(v))
\in C^1_{BC}([0,1])$. Moreover, the functions $c_1,c_2\colon C^0([0,1]) \to \R$ are continuous.
\end{lemma}
\begin{proof}
Put $u = K(v,c_1,c_2)$. Then
\[
u(0) = c_1, \quad \quad u(1) = c_1 + \int_0^1 \psi(\int_0^\tau v(s) \, ds + c_2)\, d\tau
\]
and
\[
\dot{u}(0) = \psi(c_2), \quad \quad \dot{u}(1) = \psi(\int_0^1 v(s) \, ds + c_2)
\]
For each of the boundary conditions we define a suitable function $G\colon C^0([0,1])\times\R\to\R$
and use Lemma \ref{lemma:implicitFunction} to obtain the statement.

\textbf{Case \eqref{BC:Dirichlet-nh}:}\\
In this case $c_1$ is equal to $A$. Define $G\colon C^0([0,1]) \times \R \to \R$ by
\[
G_1(v,c) = A + \int_0^t \psi(\int_0^\tau v(s) \, ds + c)\, d\tau
\]

\textbf{Case  \eqref{BC:S-L_I-nh}:}\\
From the first equation we get $c_1 = -\frac{A}{\alpha} + \frac{\beta}{\alpha}\psi(c_2)$. Now the
second equation leads to the definition of $G$:
\[
G_2(v,c) := a[ -\frac{A}{\alpha} + \frac{\beta}{\alpha}\psi(c) + \int_0^1 \psi(\int_0^\tau v(s) \,
ds + c)\, d\tau)]
+  \psi(\int_0^1 v(s) \, ds + c)
\]
It is easy to check that the functions $G_i$ satisfy the assumptions of Lemma
\ref{lemma:implicitFunction}.
Therefore both $c_1$ and $c_2$ depend continuously on $v$.
\end{proof}

By Lemma \ref{lemma:functionK}, for given boundary conditions (BC) we have a well defined continuous
function $K\colon C^0([0,1]) \to C_{BC}^1([0,1])$ given by
\[
K(v) = \hat{K}(v,c_1(v),c_2(v)).
\]
Note that the image of $K$ is contained in $C^2$ and, since the inclusion of $C^2$ into $C^1$ is
compact,
so is $K$. Define $N\colon C^1_{BC}([0,1]) \to C^0([0,1])$ by
\[
N(u)(t) = f(t,u,\dot{u}).
\]
\begin{lemma}\label{lemma:fixedPoint}
If $u$ is a fixed point of the composition $K \circ N$ then $u$ is a solution to \eqref{eq:problem}.
\end{lemma}
The proof is straightforward if we notice that $L_\varphi$ is well defined on the image of $K$ and
$L_\varphi(K(v)) = v$.
\begin{remark}
$L_\varphi$ is not well defined on the whole of $C^1_{BC}$ as can be seen in the case of the
Laplacian.
\end{remark}

Instead of looking for fixed points of $K\circ N$ one can look for zeros of $Id - K \circ N$. For
this we will use the Leray-Schauder degree and its homotopical invariance. Consider the homotopy
$H\colon [0,1] \times C^1 \to C^1$ given by
\[
H(\lambda,u) = Id - K(\lambda N(u)).
\]
Let $l = K(0)$ be a unique linear function satisfying boundary conditions. If $r > |l|$ then
\[
\deg ((H(0,\cdot),D(r)) = 1
\]
where $D(r) \subset C^1$ is a closed disc of radius $r$. If we prove that there exists $r > |l|$
such that $H(\lambda,u) \neq 0$ for any $\lambda \in [0,1]$ and any $u$ with $|u| = r$ then also
$\deg ((H(1,\cdot),D(r)) = 1$. This would prove that $Id - K \circ N = H(1,\cdot)$ has a zero.

Since $H(\lambda,u) = 0$ if and only if $u$ is a solution to the boundary value problem

\begin{equation}
\label{eq:problem:lambda}
\tag{$P_\lambda$}
\begin{cases}
\frac{d}{dt}\Phi'(\dot{u})=\lambda f(t,u,\dot{u}),\ t\in[0,1]
\\
u\in (BC)
\end{cases}
\end{equation}

we are left only to prove the following.
\begin{lemma}[apriori bounds]
If $u\in C^1$ is a solution to the problem \eqref{eq:problem:lambda} then there exists a constant
$r>0$, independent of $\lambda$ and $u$,  such that
\[
\|u\|_{C^1} \leq r.
\]
\end{lemma}

Next section is devoted to proof of this lemma.

\begin{remark}
In \cite{KelTer13,MaZhaLiu15}, authors use homotopy $H(\lambda,u) = \lambda K(N(u)) + (1-\lambda)l$.
Although $H(\lambda,u)$ satisfies boundary conditions for every $\lambda$, fixed points of
$H(\lambda,\cdot)$ does not have to be solutions to the parametrized problem, as claimed by the
authors.
\end{remark}

\section{A priori bounds}

We start by noticing that if $u$ is a $C^1$ solution to the problem \eqref{eq:problem:lambda} then
$u \in C^2$. Indeed $\dot{u}$ reads
\[
\dot{u}(t) = \psi\left( \int_0^t \lambda f(\tau,u,\dot{u}) \,d\tau + c \right)
\]
and by the assumption \ref{asm:Phi:psi_differentiable} and the continuity of $f$ it is continuously
differentiable.

The next lemma is an adaptation of Theorem 3.3 \cite{FriORe88}.

\begin{lemma}
    \label{lem:bound_for_u_2}
If $|u|$ achieves its maximum at $t_0\in (0,1)$ then
\[
|u(t)|\leq R, \text{ for $t\in [0,1]$}
\]
\end{lemma}
\begin{proof}
Suppose that $|u|$ achieves its maximum at $t_0\in (0,1)$. We can assume that $u(t_0)>R$. In the
case $u(t_0)\leq -R$ the proof is similar. Since $t_0\in (0,1)$, $\dot{u}(t_0)=0$. Let $t\in [0,1]$,
then
\begin{multline*}
\int_{t_0}^t
    (t-\sigma) u(\sigma) \frac{d}{d \tau} \varphi(\dot{u})(\tau)\Big\vert_{\tau=\sigma}
    \,d\sigma
=
    t \int_{t_0}^t u\frac{d}{d\tau}\varphi(\dot{u}) \,d\sigma
    -
    \int_{t_0}^t \sigma u\frac{d}{d\tau}\varphi(\dot{u})  \,d\sigma
= \\ =
    t\left(
        u\varphi(\dot{u})\Big\vert_{t_0}^t -  \int_{t_0}^t \dot{u}\varphi(\dot{u}) \,d\sigma
    \right)
    -
    \left(
        \sigma u\varphi(\dot{u})\Big\vert_{t_0}^t
        -
        \int_{t_0}^t (u+\sigma \dot{u})\varphi(\dot{u}) \,d\sigma
    \right)
= \\ =
    t u(t) \varphi(\dot u(t))
    -
    t \int_{t_0}^t \dot{u}\varphi(\dot{u})  \,d\sigma
    -
    t u(t) \varphi(\dot u(t))
    +
    \int_{t_0}^t u \varphi(\dot u) \,d\sigma
    +
    \int_{t_0}^t \sigma \dot{u}\varphi(\dot{u}) \,d\sigma
= \\ =
    \int_{t_0}^t u \varphi(\dot u) \,d\sigma
    +
    \int_{t_0}^t (\sigma-t) \dot{u}\varphi(\dot{u}) \,d\sigma
\end{multline*}
Hence, using \eqref{eq:problem:lambda},
\begin{equation*}
\int_{t_0}^t
    (t-\sigma)  \Big(
     \lambda u(\sigma) f\big(\sigma,u(\sigma),\dot{u}(\sigma)\big)
        +
        \dot{u}(\sigma)\varphi\big(\dot{u}(\sigma)\big)
    \Big)\,d\sigma
=
\int_{t_0}^t u(\sigma) \varphi(\dot u(\sigma)) \,d\sigma
\end{equation*}

Note that for $0<\lambda\leq 1$, $xf(t,x,0)>0$, $|x|>R$ implies $\lambda xf(t,x,0)>0$, $|x|>R$.
Thus, by the assumption \ref{asm:f:geometric}, $\lambda u(t_0)f(t_0,u(t_0),0)>0$. The continuity
of $f$, $u$ and $\dot{u}$ implies that there exists a neighborhood $N$ of $(t_0,u(t_0),0)$ such that
\[
\lambda u(t) f(t,u(t),\dot{u}(t))>0\quad \text{ for $(t,u(t),\dot{u}(t)\in N.$}
\]

Since $u\in C^1$ and achieves its maximum at $t_0$, there exist $t_0^-$ and $t_0^+$ such
that
\begin{itemize}
\item $u(t)>R$ for $t\in (t_0^-,t_0^+)$,
\item $\dot{u}$ is non-negative on $(t_0^-,t_0]$,
\item $\dot{u}$ is non-positive on $(t_0^-,t_0]$.
\end{itemize}
Hence $\varphi(\dot{u}(t))\geq 0$ for $t\in (t_0^-,t_0]$ and $\varphi(\dot{u}(t))\leq 0$ for $t\in
[t_0,t_0^+)$.
This implies that
\begin{equation*}
\int_{t_0}^t (t-\sigma)\dot{u}(\sigma) \varphi(\dot u(\sigma)) \,d\sigma \geq 0
\text{ for $t\in (t_0^-,t_0^+)$}
\end{equation*}
and
\begin{equation*}
\int_{t_0}^t u(\sigma) \varphi(\dot u(\sigma)) \,d\sigma \leq 0
\text{ for $t\in (t_0^-,t_0^+).$}
\end{equation*}
It follows that for $t$ close to $t_0$
\[
0<
\int_{t_0}^t
    (t-\sigma)  \Big(
   \lambda u(\sigma) f\big(\sigma,u(\sigma),\dot{u}(\sigma)\big)
        +
        \dot{u}(\sigma)\varphi\big(\dot{u}(\sigma)\big)
    \Big)\,d\sigma
=
\int_{t_0}^t u(\sigma) \varphi(\dot u(\sigma)) \,d\sigma
\leq 0,
\]
a contradiction. Thus $u(t_0)\leq R$.
\end{proof}

\begin{lemma}
There exists a constant $r_0>0$, independent of $u$ and $\lambda$ such that
\[
|u(t)|\leq r_0 \text{ for all $t\in [0,1]$}
\]
\end{lemma}
\begin{proof}

If $\lambda=0$ then the problem \eqref{eq:problem:lambda} has a unique solution and thus
$|u(t)|\leq C$ for some constant $C\geq 0$. Let $0<\lambda\leq 1$.

Assume that $u$ satisfies \eqref{BC:Dirichlet-nh}. If $|u|$ achieves its
maximum at $t_0=0$ (respectively $t_0=1$) then $|u(t)|\leq |A|$ (resp. $|u(t)|\leq |B|$). If the
maximum
is achieved in $t_0\in (0,1)$ then by Lemma \ref{lem:bound_for_u_2} we get $|u(t)|\leq R$. Hence
\[
|u(t)|\leq K_0 = \max\{R,|A|,|B|\}.
\]

Assume that $u$ satisfies \eqref{BC:S-L_I-nh}. If $|u|$ has its maximum value at $0$ then
$u(0)\dot{u}(0)\leq 0$. The boundary conditions give
\[
u(0)(A+\alpha u(0)) = \beta u(0)\dot{u}(0)\leq 0
\]
and consequently $|u(0)|\leq |A/\alpha|$. A similar argument shows that $|u(1)|\leq |B/a|$.
If the maximum is achieved in $t_0\in (0,1)$ then by Lemma \ref{lem:bound_for_u_2} we get
$|u(t)|\leq M$. Finally
\[
|u(t)|\leq r_0 = \max\{M,|A/\alpha|,|B/a|\}.
\]
\end{proof}

Now we provide bounds for $\dot{u}$. The proof of the following theorem is based on \cite{FriORe88}.

\begin{lemma}
There exists a constant $r_1>0$ (depending only on $r_0$, $A$, $B$ and $\Phi)$ such that
\[
|\dot{u}(t)|\leq r_1 \text{ for all $t\in [0,1]$}
\]
\end{lemma}
\begin{proof}
Since we have obtained a priori bounds $|u(t)|\leq r_0$, it is easy to observe that there exists
a constant $C\geq 0$ independent of $\lambda$ and $u$, such that
\[
|\dot{u}(t_0)|\leq C
\]
for some $t_0\in [0,1]$. The point $t_0$ belongs to an interval $[\mu,\nu]\subset [0,1]$ such that
the
sign of  $\dot{u}(t)$ does not change in $[\mu,\nu]$ and
$\dot{u}(\mu)=\dot{u}(t_0)$ and/or $\dot{u}(\nu)=\dot{u}(t_0)$.

Assume that $\dot{u}(\mu)=\dot{u}(t_0)$ and $\dot{u}(t)\geq 0$ for every $t\in[\mu,\nu]$. The other
cases are treated similarly and the same bound is obtained.

Denote by $S_0$, $T_0$ the upper bounds of $A$ and $B$ respectively on $[0,1]\times
[-r_0,r_0]$. Since
\[
|\lambda f(t,u,\dot{u})|\leq S_0 (\Phi'(\dot{u})\dot{u}-\Phi(\dot{u}))+T_0,
\]
we have
\[
\int_\mu^t
\frac{
    S_0\,\dot{u}\, \left|\frac{d}{dt}\Phi'(\dot{u})\right|
} {
    S_0(\Phi'(\dot{u})\dot{u}-\Phi(\dot{u}))+T_0
}\,d\tau
\leq
S_0 \int_\mu^t \,\dot{u}\, d\tau
\leq 2 S_0 r_0
\]
For $\mu\leq \tau \leq t$, we have

\begin{multline*}
\left( \Phi'(\dot{u}(\tau))\dot{u}(\tau)-\Phi(\dot{u}(\tau)) \right)
-
\left( \Phi'(\dot{u}(\mu))\dot{u}(\mu)-\Phi(\dot{u}(\mu)) \right)
=\\=
\int_\mu^t \frac{d}{dt}\left( \Phi'(\dot{u}(\tau))\dot{u}(\tau)-\Phi(\dot{u}(\tau))
\right)|_{t=\sigma} \,dt
=
\int_\mu^\tau \dot{u}\,\frac{d}{dt}\Phi'(\dot{u}) \,d\sigma.
\end{multline*}
There exists $C_0\geq 0$ such that
$0 \leq S_0\left(\Phi'(\dot{u}(\mu))\dot{u}(\mu)-\Phi(\dot{u}(\mu))\right)+T_0 \leq C_0$.
Hence,
\[
0 \leq
S_0 \left( \Phi'(\dot{u}(\tau))\dot{u}(\tau)-\Phi(\dot{u}(\tau)) \right)+T_0
\leq
S_0  \int_\mu^\tau \dot{u}\,\left|\frac{d}{dt}\Phi'(\dot{u})\right| \,d\sigma + C_0 + T_0.
\]
Set $g(\tau) = S_0 \int_\mu^\tau \dot{u}\,\left|\frac{d}{dt}\Phi'(\dot{u})\right| \,d\sigma + C_0$,
then integration by substitution yields
\[
\log\left(\frac{g(t)+T_0}{C_0+T_0}\right)
=
\int_{C_0}^{g(t)}\frac{1}{x +T_0}\,dx
=
\int_\mu^t
\frac{S_0\,\dot{u}\,\frac{d}{dt}\Phi'(\dot{u})}{g(\tau) +T_0}\,d\tau
\leq 2 S_0 r_0.
\]
Thus
\[
g(t)\leq (T_0+C_0) e^{2 S_0 r_0} - T_0
\]
and by \ref{asm:Phi:nabla2}
\[
(k_\Phi-1)\Phi(\dot{u}(t))
\leq
\Phi'(\dot{u}(t))\dot{u}(t)-\Phi(\dot{u}(t))
\leq
\frac{1}{S_0} ((T_0+C_0) e^{2 S_0 r_0} - T_0)
\]
The last inequality gives $|\dot{u}(t)| \leq r_1$ for all $t\in [0,1]$.
\end{proof}


\bibliographystyle{elsarticle-num}
\bibliography{final_for_arxive}

\end{document}